\documentclass[12pt]{article}
\usepackage{amssymb, enumerate}
\usepackage{latexsym}
\usepackage{url}

\usepackage{amsmath,amssymb}

\newtheorem{theorem}{Theorem}[section]
\newtheorem{lemma}[theorem]{Lemma}
\newtheorem{corollary}[theorem]{Corollary}
\newtheorem{proposition}[theorem]{Proposition}

\newenvironment{proof}{\normalsize {\sc Proof}.}{{\hfill $\Box$%
 \hskip - \parfillskip\bigskip}}

\newcommand{\Res}{\mathop{\rm Res}\nolimits}

\newcommand{\Irr}{\mathop{\rm Irr}\nolimits}

\newcommand{\Aut}{\mathop{\rm Aut}\nolimits}

\newcommand{\NN} {\mathbb{N}}
\newcommand{\ZZ} {\mathbb{Z}}

\def\bigcp{\mathop{\mathchoice 
 {\hbox{\sf\Large\lower 0.1\baselineskip\hbox{Y}}}%
 {\hbox{\sf\large\lower 0.1\baselineskip\hbox{Y}}}%
 {\hbox{\sf\normalsize\lower 0.1\baselineskip\hbox{Y}}}%
 {\hbox{\sf\tiny\lower 0.1\baselineskip\hbox{Y}}}%
}}
\def\bigtimes{\mathop{\mathchoice 
 {\hbox{\sf\Large\lower 0.1\baselineskip\hbox{X}}}%
 {\hbox{\sf\large\lower 0.1\baselineskip\hbox{X}}}%
 {\hbox{\sf\normalsize\lower 0.1\baselineskip\hbox{X}}}%
 {\hbox{\sf\tiny\lower 0.1\baselineskip\hbox{X}}}%
}}

\def\Sym(#1){\mathop{\rm Sym}(#1)}
\def\Sym(#1){S_{#1}}
\def\diag(#1){\mathop{\rm diag}(#1)}

\newenvironment{enumerate*}{%
 \begin{enumerate}%
 }%
 {\end{enumerate}}

\begin{document}

\title{Morita equivalence classes of $2$-blocks of defect three}

\author{Charles W. Eaton}

\date{22nd September 2014}
\maketitle


\begin{abstract}
We give a complete description of the Morita equivalence classes of blocks with elementary abelian defect groups of order $8$ and of the derived equivalences between them. A consequence is the verification of Brou\'e's abelian defect group conjecture for these blocks. It also completes the classification of Morita and derived equivalence classes of $2$-blocks of defect at most three defined over a suitable field.
\end{abstract}


\section{Introduction}

Throughout let $k$ be an algebraically closed field of prime characteristic $\ell$ and let $\mathcal{O}$ be a discrete valuation ring with residue field $k$ and field of fractions $K$ of characteristic zero. We assume that $K$ is large enough for the groups under consideration. We consider blocks $B$ of $\mathcal{O}G$ with defect group $D$.

We are concerned with the description of the Morita and derived equivalence classes of (module categories for) blocks of finite groups with a given defect group $D$. We briefly review progress on this problem to date. If $D$ is an abelian $p$-group whose automorphism group is a $p$-group, then any block with defect group $D$ must be nilpotent and so Morita equivalent to $\mathcal{O} D$ (see~\cite{ku80} and~\cite{puig88}). There are many other examples of $p$-groups for which it has been proved that every fusion system is nilpotent, but we do not list these here. If $D$ is cyclic, then the Morita equivalence classes can be characterised in terms of Brauer trees, in work going back to Brauer and Dade (see~\cite{li96}). In a series of papers Erdmann characterises the Morita equivalence classes of tame blocks defined over $k$ except when $D$ is generalised quaternion and $B$ has two simple modules (see~\cite{er90}), although the problem remains open for blocks defined over $\mathcal{O}$. The (three) Morita equivalence classes of blocks defined over $\mathcal{O}$ with defect group $C_2 \times C_2$ are determined in~\cite{li94}. When $D=\langle x,y: x^{2^r}=y^{2^s}=[x,y]^2=[x,[x,y]]=[y,[x,y]]=1\rangle$, where $r \geq s \geq 1$ (nonmetacylic minimal nonabelian $2$-group), the Morita equivalence classes of blocks are determined in~\cite{sa11} and~\cite{eks12}. When $D$ is a homocyclic $2$-group, the Morita equivalence classes of blocks are determined in~\cite{ekks14}.

In this paper we use the classification given in~\cite{ekks14} to completely determine the Morita and derived equivalence classes of blocks defined over $\mathcal{O}$ with defect group $D \cong C_2 \times C_2 \times C_2$. As a consequence it follows that Brou\'e's abelian defect group conjecture holds for blocks of defect three. We also note that this completes the classification of Morita equivalence classes of $2$-blocks of defect at most three, for blocks defined over $k$. Blocks with elementary abelian defect groups of order $8$ have already been studied in~\cite{kkl12}, where it is shown that Alperin's weight conjecture and the isotypy version of Brou\'e's abelian defect group conjecture hold for these blocks. The results of~\cite{kkl12} are needed here, in particular to achieve Morita equivalences over $\mathcal{O}$ rather than $k$.

Before stating the main theorem, we recall the definition of the inertial quotient of $B$. Let $b_D$ be a block of $\mathcal{O}DC_G(D)$ with Brauer correspondent $B$, and write $N_G(D,b_D)$ for the stabilizer in $N_G(D)$ of $b_D$ under conjugation. Then the \emph{inertial quotient} of $B$ is $E=N_G(D,b_D)/DC_G(D)$, an $\ell'$-group unique up to isomorphism.

\begin{theorem}
\label{maintheorem}
Let $B$ be a block of $\mathcal{O}G$, where $G$ is a finite group. If $B$ has defect group $D$ isomorphic to $C_2 \times C_2 \times C_2$, then $B$ is Morita equivalent to the principal block of precisely one of the following:

(i) $D$;

(ii) $D \rtimes C_3$;

(iii) $C_2 \times A_5$, and the inertial quotient is $C_3$.

(iv) $D \rtimes C_7$;

(v) $SL_2(8)$, and the inertial quotient is $C_7$;

(vi) $D \rtimes (C_7 \rtimes C_3)$;

(vii) $J_1$, and the inertial quotient is $C_7 \rtimes C_3$;

(viii) $^2G_2(3) \cong \Aut(SL_2(8))$, and the inertial quotient is $C_7 \rtimes C_3$;

\medskip

Blocks are derived equivalent if and only if they have the same inertial quotient.
\end{theorem}

 A block with defect group $C_2 \times C_2 \times C_2$ cannot be Morita equivalent to a block with non-isomorphic defect group. This is since Morita equivalence preserves defect and (i) $2$-blocks of defect three with abelian defect groups other than $C_2 \times C_2 \times C_2$ must be nilpotent (and so Morita equivalent to the group algebra of a defect group), (ii) $2$-blocks of defect three with nonabelian defect groups have five irreducible characters (whilst the number is eight for blocks with defect group $C_2 \times C_2 \times C_2$).

\begin{corollary}
Brou\'e's abelian defect group conjecture holds for all $2$-blocks with defect at most three. That is, let $B$ be a block of $\mathcal{O}G$ for a finite group $G$ with defect group $D$ of order dividing $8$, and let $b$ be the unique block of $\mathcal{O}N_G(D)$ with Brauer correspondent $B$. Then $B$ and $b$ have derived equivalent module categories.
\end{corollary}

\begin{proof}
If a defect group $D$ are isomorphic to $C_2$, $C_4$, $C_4 \times C_2$ or $C_8$, then the block is nilpotent, in which case the conjecture holds automatically since $\Aut(D)$ is a $2$-group. If $D \cong C_2 \times C_2$, then the result follows from~\cite{li94}. Suppose that $D \cong C_2 \times C_2 \times C_2$. By Theorem \ref{maintheorem} the derived equivalence class of $B$ is uniquely determined by the number $l(B)$ of irreducible Brauer characters. Since every block with defect group $D$ has eight irreducible characters, it is a consequence of Brauer's second main theorem that $l(B)=l(b)$ and the result follows.
\end{proof}

Note that we do not prove that there are splendid derived equivalences of blocks.

\begin{corollary}
Let $B$ be a block with defect group $D \cong C_2 \times C_2 \times C_2$. Then $B$ has Loewy length $LL(B)$ equal to $4$, $6$ or $7$.
\end{corollary}

\begin{proof}
By Theorem \ref{maintheorem} it suffices to consider cases (i)-(viii) in the notation of that theorem. In cases (i), (ii), (iv) and (vi), where $D \lhd G$ and $[G:D]$ is odd, we have that $LL(B)=LL(kD)=4$, by~\cite[4.1]{km01}. In case (iii) $LL(B)=6$. In the remaining cases $LL(B)=7$ by~\cite{al79} and~\cite{lm78}, again using~\cite[4.1]{km01}.
\end{proof}

\begin{corollary}
Let $B$ be a $2$-block of defect at most $3$, then the Cartan invariants of $B$ are at most the order of a defect group.
\end{corollary}

Of course the above does not hold in generality.

Since we now have a complete list of Cartan matrices (up to ordering of the simple modules), and indeed the decomposition matrices, for $2$-blocks of defect at most $3$, it would be interesting to look for possible concrete restrictions on Cartan matrices.

\section{Quoted results}

The following proposition will be used when considering automorphism groups of simple groups. It gathers together two propositions from~\cite{kkl12}, which in turn gathers results from~\cite{da73} and~\cite{ku90}.

\begin{proposition}
\label{innerautos}
Let $\ell$ be any prime and let $G$ be a finite group and $N \lhd G$ with $[G:N]=w$ a prime not equal to $\ell$. Let $b$ be a $G$-stable $\ell$-block of $\mathcal{O} N$. Then either each block of $\mathcal{O}G$ covering $b$ is Morita equivalent to $b$, or there is a unique block of $\mathcal{O}G$ covering $b$. In the former case, $B$ and $b$ have isomorphic inertial quotient.

\end{proposition}

\begin{proof}
Note that the group $G[b]$ of elements of $G$ acting as inner automorphisms on $b$ is a normal subgroup of $G$ containing $N$. If $G[b]=G$, then each block of $G$ covering $b$ is source algebra equivalent to $b$ by~\cite[2.2]{kkl12}, and has inertial quotient isomorphic to that of $b$ by~\cite[3.4]{kkl12}. If $G[b]=N$, then there is a unique block of $G$ covering $b$ by~\cite[2.3]{kkl12}.
\end{proof}

The following is a distillation of those results in~\cite{kp90} which are relevant here.

\begin{proposition}[\cite{kp90}]
\label{kp}
Let $G$ be a finite group and $N \lhd G$. Let $B$ be a block of $\mathcal{O} G$ with defect group $D$ covering a $G$-stable nilpotent block $b$ of $\mathcal{O} N$ with defect group $D \cap N$. Then there is a finite group $L$ and $M \lhd L$ such that (i) $M \cong D \cap N$, (ii) $L/M \cong G/N$, (iii) there is a subgroup $D_L$ of $L$ with $D_L \cong D$ and $D_L \cap M \cong D \cap N$, and (iv) there is a
a central extension $\tilde{L}$ of $L$ by an $\ell'$-group, and  a
 block $\tilde{B}$ of  $\mathcal{O} \tilde L$   which is Morita equivalent to $B$ and has defect group $\tilde{D} \cong D$.
\end{proposition}

\begin{proposition}[\cite{wa94}]
\label{watanabe}
Let $B$ be an $\ell$-block of $\mathcal{O} G$ for a finite group $G$ and let $Z \leq O_\ell(Z(G))$. Let $\bar B$ be the unique block of $\mathcal{O}(G/Z)$ corresponding to $B$. Then $B$ is nilpotent if and only if $\bar B$ is nilpotent.
\end{proposition}

\begin{proposition}[\cite{ekks14}]
\label{classification}
Let $B$ be a block of $\mathcal{O}G$ for a quasisimple group $G$ with elementary abelian defect group $D$ of order $8$. Then one of the following occurs:

(i) $G \cong SL_2(8)$ and $B$ is the principal block;

(ii) $G \cong$ $^2G_2(q)$, where $q=3^{2m+1}$ for some $m \in \NN$, and $B$ is the principal block;

(iii) $G \cong J_1$ and $B$ is the principal block;

(iv) $G \cong Co_3$ and $B$ is the unique non-principal $2$-block of defect $3$;

(v) $G$ is of type $D_n(q)$ or $E_7(q)$ for some $q$ of odd prime power order, $O_2(G)=1$ and $B$ is Morita equivalent to the principal block of $C_2 \times A_5$ or of $C_2 \times A_4$.

(vi) $|O_2(G)|=2$ and $D/O_2(G)$ is a Klein four group;

(vii) $B$ is nilpotent.
\end{proposition}

\begin{lemma}
\label{normaldefect}
Let $B$ be a block of $\mathcal{O}G$ for a finite group $G$ with normal defect group $D \cong C_2 \times C_2 \times C_2$. Then $B$ is Morita equivalent to $\mathcal{O} (D \rtimes E)$, where $E$ has odd order and acts faithfully on $D$.
\end{lemma}

\begin{proof}
This is well known, but may be obtained for instance by applying Proposition \ref{kp} and noting that the inertial quotient is one of $1$, $C_3$, $C_7$ and $C_7 \rtimes C_3$, each having trivial Schur multiplier.
\end{proof}

\section{Preliminary results}

\begin{proposition}
\label{ree}
Let $N=$ $^2G_2(q)$, where $q=3^{2m+1}$ for some $m \in \NN \cup \{0\}$, and $N \leq G \leq \Aut (N)$. Let $b$ be the principal $2$-block of $\mathcal{O}N$. Then every block of $\mathcal{O}G$ covering $b$ is source algebra equivalent to $b$. Further, each of these blocks shares a defect group with $b$ and has isomorphic inertial quotient.
\end{proposition}

\begin{proof}
$G/N$ is cyclic of odd order. Let $N=G_0 \leq G_1 \leq \cdots \leq G_n=G$, with each $|G_{i+1}/G_i|$ prime. By~\cite{wa66} $b$ has defect groups of the form $C_2 \times C_2 \times C_2$ and irreducible character degrees occurring with multiplicity either one or two, so that each irreducible character is $G$-stable. Since $[G:N]$ is odd each block of $\mathcal{O}G_i$ covering $b$ shares a defect group with $b$. By~\cite{kkl12}, every block with defect group $C_2 \times C_2 \times C_2$ (in particular $b$ and every block of $\mathcal{O}G_i$ covering it) has precisely eight irreducible characters, and it follows that for each $i$ there are $[G_{i+1}:N]$ $2$-blocks of $\mathcal{O}G_{i+1}$ covering $b$, and amongst these there $[G_{i+1}:G_i]$ blocks of $\mathcal{O}G_{i+1}$ covering each such block of $\mathcal{O}G_i$. It follows from Proposition \ref{innerautos} that each block of $\mathcal{O}G_i$ covering $b$ is source algebra equivalent to $b$. That the blocks have isomorphic inertial quotient follows from~\cite[3.4]{kkl12}.
\end{proof}

\begin{proposition}
\label{typeC3}
Let $G$ be a finite group and $N \lhd G$ with $[G:N]$ an odd prime. Let $b$ be a $G$-stable block of $\mathcal{O}N$ with defect group $C_2 \times C_2 \times C_2$ and inertial quotient $C_3$. Suppose that $l(b)=3$. Let $B$ be a block of $\mathcal{O}G$ covering $b$. Then either $B$ is source algebra equivalent to $b$ or nilpotent. In the former case $B$ has inertial quotient $C_3$ and $[G:N]=3$.
\end{proposition}

\begin{proof}
By~\cite{kkl12} we have $l(B) \leq 7$. Suppose first that $[G:N] \geq 5$. Since we are assuming that $l(b)=3$, there cannot be a unique block of $\mathcal{O}G$ covering $b$ (since each irreducible Brauer character of $b$ is $G$-stable and so the total number of irreducible Brauer characters in blocks covering $B$ is at least $15$), so by Proposition \ref{innerautos} $B$ is source algebra equivalent to $b$ and has the same inertial quotient.

Suppose now that $[G:N]=3$. If every irreducible Brauer character of $b$ is $G$-stable, in which case again by Proposition \ref{innerautos} $B$ is source algebra equivalent to $b$ and has the same inertial quotient. If the three irreducible Brauer characters are permuted transitively, then $l(B)=1$, so that by~\cite{kkl12} $B$ is nilpotent.
\end{proof}

The following is a strengthening of a special case of the main result of~\cite{kk96}, which is only known to hold for blocks defined over $k$.

\begin{proposition}
\label{kkext}
Let $G$ be a finite group and $N \lhd G$ and let $C$ be a $G$-stable block of $\mathcal{O}N$ covered by a block $B$ of $\mathcal{O}G$ with elementary abelian defect group $D$ of order $8$. Write $P=N \cap D$ and suppose that $D=P \times Q$ for some $Q$ of order $2$ such that $G=N \rtimes Q$. Then $B \cong C \otimes_\mathcal{O} \mathcal{O}Q$. In particular $B$ and the block $C \otimes_\mathcal{O} \mathcal{O}Q$ of $\mathcal{O}(N \times Q)$ are Morita equivalent.
\end{proposition}

\begin{proof}
Write $Q=\langle x \rangle$. As noted in~\cite{kk96} $B$ and $C$ share a block idempotent $e$, so that $B$ is a crossed product of $C$ with $Q$ and it suffices to find a graded unit of $Z(B)$ of degree $x$ and order two. We do this by exploiting the existence of a perfect isometry as shown in~\cite[5.1]{kkl12}, although we must show that this perfect isometry satisfies additional properties. Part of the proof follows that of~\cite[5.1]{kkl12}, and we take facts from there without explicit further reference. Note however that for convenience we use a different labeling of the irreducible characters.

Denote by $E$ the inertial quotient of $B$, so that $|E|=1$ or $3$. If $|E|=1$, then $B$ is nilpotent and the result follows from~\cite{puig88}. Hence we may assume that $|E|=3$ and $E$ acts faithfully on $D$. Write $H=D \rtimes E$. Then $Q \leq Z(H)$ and so $H=(P \rtimes E) \times Q$.

By~\cite{la81} we have $k(B)=8$. Label the irreducible characters $\theta_i$ of $H$ so that $\theta_1,\ldots,\theta_4$ have $Q$ in their kernel, $\theta_1(1)=\theta_2(1)=\theta_3(1)=1$, $\theta_4(1)=3$ and $\theta_i (g)=\theta_{i-4}(g)$ for all $i=5,\ldots,8$ and all $g \in P \rtimes E$. We have $\theta_i (x)=-\theta_{i}(1)$ for $i=5,\ldots,8$. Similarly label the irreducible characters $\chi_1, \ldots, \chi_8$ of $B$ so that $\Res_N^G(\chi_i)=\Res_N^G(\chi_{i-4})$ for all $i=5,\ldots,8$. Note that $\chi_i (x)=-\chi_{i-4}(x)$ for all $i=5,\ldots,8$.

There is a stable equivalence of Morita type between $\mathcal{O}H$ and $B$, leading to an isometry $L^0(H,\mathcal{O}H) \cong L^0(G,B)$ between the groups of generalised characters vanishing on $2$-regular elements. $L^0(H, \mathcal{O}H)$ is generated by $$\{\theta_1-\theta_5, \theta_2-\theta_6, \theta_3-\theta_7, \theta_4-\theta_8, \theta_1+\theta_2+\theta_3-\theta_4\}.$$ We claim that if $\chi_i-\chi_j \in L^0(G,B)$, then $|i-j|=4$. For suppose that $\chi_i(g)=\chi_j(g)$ for all $g \in G$ of odd order. Then $\Res_N^G(\chi_i)$ and $\Res_N^G(\chi_2)$ are irreducible characters of $C$ agreeing on $2$-regular elements. Noting that $C$ is not nilpotent, and that $C$ has decomposition matrix that of the principal $2$-block of $A_4$ or $A_5$, it follows that $\Res_N^G(\chi_i)=\Res_N^G(\chi_2)$ and the claim follows.

Hence the isometry takes elements of the form $\theta_i-\theta_{i-4}$ to elements of the form $\delta_j(\chi_j-\chi_{j-4})$. Now the isometry extends to a perfect isometry $I:\ZZ \Irr(H) \rightarrow \ZZ \Irr(B)$, and we have seen that $I(\theta_i)(g)=I(\theta_{i-4})(g)$ for every $i=5,\ldots,8$ and every $g \in N$.

Following~\cite{br90} $I$ induces an $\mathcal{O}$-algebra isomorphism $I^0:Z(\mathcal{O}H)\rightarrow Z(B)$ with $I^0(x)= \frac{1}{|H|} \sum_{g \in G} \mu (g^{-1},x)g$, where $\mu (g,h)= \sum_{i=1}^8 \theta_i (h) I(\theta_i)(g)$ for $g \in G$ and $h \in H$. We will show that $I^0(x) = ax$ for some $a \in \mathcal{O}N$, i.e., that $\mu (g,x)=0$ whenever $g \in N$. Then $I^0(x)$ will be the required graded unit of $Z(B)$ of degree $x$ and order two.

Let $g \in N$. Then $$\mu(g,x) =  \sum_{i=1}^8 \theta_i (x) I(\theta_i)(g) = \sum_{i=5}^8 \theta_{i-4}(1)\left( I(\theta_{i-4})(g) - I(\theta_{i})(g) \right) = 0$$ and we are done.
\end{proof}


\section{Proof of the main theorem}

We prove Theorem \ref{maintheorem}.

\begin{proof}
Let $B$ be a block of $\mathcal{O}G$ for a finite group $G$ with defect group $D \cong C_2 \times C_2 \times C_2$ with $[G:Z(G)]$ minimised such that $B$ is not Morita equivalent to any of (i)-(viii). By minimality and the first Fong reduction $B$ is quasiprimitive, that is, for every $N \lhd G$ each block of $\mathcal{O}N$ covered by $B$ is $G$-stable. By Proposition \ref{kp} if $N \lhd G$ and $B$ covers a nilpotent block of $\mathcal{O}N$, then $N \leq Z(G) O_2(G)$. In particular $O_{2'}(G) \leq Z(G)$

Following~\cite{as00} write $E(G)$ for the \emph{layer} of $G$, that is, the central product of the subnormal quasisimple subgroups of $G$ (the \emph{components}). Write $F(G)$ for the Fitting subgroup, which in our case is $F(G)=Z(G)O_2(G)$. Write $F^*(G)=F(G)E(G) \lhd G$, the generalised Fitting subgroup, and note that $C_G(F^*(G)) \leq F^*(G)$. Let $b$ be the (unique) block of $\mathcal{O}F^*(G)$ covered by $B$.

Let $\overline{B}$ be the unique block of $\mathcal{O}(G/O_2(Z(G)))$ corresponding to $B$. First observe that $|O_2(Z(G))| \leq 2$, for otherwise $\overline{B}$ would have defect at most one and so would be nilpotent, which in turn would mean that $B$ would be nilpotent by Proposition \ref{watanabe}, a contradiction.

If $|O_2(G)| > 4$, then $O_2(G)=D$, a contradiction by Lemma \ref{normaldefect}. Hence $|O_2(G)| \leq 4$.

\emph{Claim.} $O_2(G) \leq Z(G)$ and $|O_2(G)| \leq 2$.

Suppose that $O_2(G) \not\leq Z(G)$ (so $|O_2(G)|=4$). If $O_2(Z(G)) \neq 1$, then $O_2(G/O_2(Z(G)))$ has order $2$ and so is central in $G/O_2(Z(G))$, from which it follows using Proposition \ref{watanabe} that $\overline{B}$, and so $B$, is nilpotent, again a contradiction. If $O_2(Z(G)) = 1$, then $F^*(G) = O_2(G) \times (Z(G)E(G))$. Since $|O_2(G)|=4$, $B$ covers a nilpotent block of $F^*(G)$ and so $F^*(G)=O_2(G)Z(G)$. But $C_G(F^*(G)) \leq F^*(G)$ and so $D \leq C_G(O_2(G)) \leq O_2(G)Z(G)$, a contradiction. Hence $O_2(G) \leq Z(G)$ (and $|O_2(G)| \leq 2$) as claimed.

Write $E(G)=L_1 *\cdots *L_t$, where each $L_i$ is a component of $G$ (arguing as above we have that $t \geq 1$). Now $B$ covers a block $b_E$ of $\mathcal{O}E(G)$ with defect group contained in $D$, and $b_E$ covers a block $b_i$ of $\mathcal{O}L_i$. Since $b_E$ is $G$-stable, for each $i$ either $L_i \lhd G$ or $L_i$ is in a $G$-orbit in which each corresponding $b_i$ is isomorphic (with equal defect). Since $B$ has defect three, it follows that if $t>1$, then $B$ covers a nilpotent block of a normal subgroup generated by components of $G$, a contradiction. Hence $t=1$. So $G$ has a unique component $L_1$, and $G/Z(G) \leq \Aut(L_1Z(G)/Z(G))$.

Suppose that $O_2(G) \not\leq [L_1,L_1]$. Then $F^*(G)=O_2(G) \times Z(G)L_1$. In this case $D \leq F^*(G)$, since otherwise $b$ would be nilpotent. Since $b$ is $G$-stable, this means $[G:F^*(G)]$ odd and so $O_2(G)$ is in fact a direct factor of $G$. By~\cite{li94} it follows that $B$ is Morita equivalent to one of (ii) or (iii), a contradiction. Hence $O_2(G) \leq [L_1,L_1]$.

We next show that $D \leq F^*(G)$. Suppose otherwise. Then since $D$ is elementary abelian we may write $D=(D \cap F^*(G)) \times Q$ for some $Q$ of order $2$ (if $Q$ were to be larger, then $B$ would cover a nilpotent block of $\mathcal{O}F^*(G)$). By the Schreier conjecture $G/F^*(G)$ is solvable. Since $b$ is $G$-stable, $DF^*(G)/F^*(G)$ is a Sylow $2$-subgroup of $G/F^*(G)$. Hence $G=H \rtimes Q$ for some $H \lhd G$. By Proposition \ref{kkext} $B \cong b \otimes_\mathcal{O} \mathcal{O}Q$ as $\mathcal{O}$-algebras. Now $b \otimes_\mathcal{O} \mathcal{O} Q$ is a block of $\mathcal{O}(H \times Q)$ with defect group $D=(D \cap H) \times Q$. Since $b$ is Morita equivalent to the principal block of $\mathcal{O}A_4$ or $\mathcal{O}A_5$, it follows that $B$ is Morita equivalent to one of (ii) or (iii). Hence $D \leq F^*(G)$. Since $[F^*(G):L_1]$ is odd, this means $D$ is also a defect group for $b_1$.

We now refer to Proposition \ref{classification}. Suppose that $L_1 \cong SL_2(8)$ and $b_1$ is the principal block. Then $G$ is $SL_2(8)$ or $\Aut(SL_2(8)) \cong SL_2(8) \rtimes C_3 \cong $ $^2G_2(3)$, leading to (v) or (viii) of the theorem.

If $L_1 \cong $ $^2G_2(3^{2m+1})$ for some $m \in \NN$, then $L_1 \leq G \leq \Aut(^2G_2(3^{2m+1}))$ and by Proposition \ref{ree} $B$ is Morita equivalent to $b_1$. By~\cite[Example 3.3]{ok97} $b_1$ is Morita equivalent to the principal block of $\mathcal{O}(^2G_2(3))$.

If $L_1 \cong J_1$ or $Co_3$, then $G=L_1$. By~\cite[1.5]{kmn11} the $2$-block of $\mathcal{O}Co_3$ of defect three is Morita equivalent to the principal block of $\mathcal{O}(^2G_2(3))$, hence we are done in this case. The principal block of $\mathcal{O}J_1$ is not Morita equivalent to that of $\mathcal{O}(^2G_2(3))$, since their decomposition matrices are not similar (see~\cite{gap} for decomposition matrices).

Suppose that $L_1$ is of type $D_n(q)$ or $E_7(q)$ and $b_1$ is Morita equivalent to the principal block of $\mathcal{O}(C_2 \times A_4)$ or $\mathcal{O}(C_2 \times A_5)$. Then $G/L_1$ is abelian and of odd order. By Proposition \ref{typeC3} $B$ is either nilpotent (a contradiction) or Morita equivalent to $b_1$, and we are done in this case.

This leaves the case that $|O_2(L_1)|=2$ and $D/O_2(L_1)$ is a Klein four group. We have shown that $O_2(N)=O_2(G)$. Recall that $\overline{B}$ is the unique block of $\mathcal{O}(G/O_2(G))$ corresponding $B$, and note that $\overline{B}$ has defect group $D/O_2(G)$. By\cite{CEKL} $\overline{B}$ is source algebra equivalent to the principal block of $\mathcal{O}A_4$ or of $\mathcal{O}A_5$. It follows from~\cite[Corollary 1.14]{pu01} that $B$ is Morita equivalent to the principal block of a central extension of $A_4$ or $A_5$ by a group of order $2$, i.e., of $C_2 \times A_4$ or $C_2 \times A_5$.

To see that the blocks in cases (i)-(viii) represent distinct Morita equivalence classes it suffices to note that they have distinct decomposition matrices.

\end{proof}

\begin{center} ACKNOWLEDGEMENTS \end{center}

I thank Markus Linckelmann for some useful correspondence on the subject of lifting Morita equivalences.

\vspace{10mm}
Charles Eaton

School of Mathematics

University of Manchester

Oxford Road

Manchester

M13 9PL

United Kingdom

charles.eaton@manchester.ac.uk
\end{document}